\documentclass[a4paper,11pt]{amsart}

\usepackage{amscd}
\usepackage{xypic}  
\usepackage{amssymb}
\usepackage{amsthm}
\usepackage{epsfig}
\usepackage{color}
\usepackage[toc,page]{appendix}
\usepackage[T1]{fontenc}
\usepackage[all,cmtip]{xy}
\usepackage{amsmath}
\usepackage{paralist}
\usepackage{comment}

\newtheorem{thm}{Theorem}[section] 

\newtheorem{lemma}[thm]{Lemma}

\newtheorem{proposition}[thm]{Proposition}

\newtheorem{example}[thm]{Example}

\newtheorem{corollary}[thm]{Corollary}
\theoremstyle{definition}

\newtheorem{remark}[thm]{Remark}

\newtheorem{definition-remark}[thm]{Definition-Remark}
 \newtheorem*{main-proposition}{Proposition 1.3}
 \newtheorem*{main-result}{Theorem}
 \newtheorem*{main-example}{Example 1.7}
\newtheorem*{main-corollary}{Corollary 2.2}

\def\Ext{{\rm{Ext}}}
\def\Ex{{\rm{Ex}}}

\def\ker{\operatorname{ker}}

\def\Im{\operatorname{Im}}

\def\c1{\operatorname{c_1}}
\def\c2{\operatorname{c_2}}

\def\kk{{\bf k}}

\def\A{{\mathcal A}}

\def\O{{\mathcal O}}
\def\I{{\mathcal J}}

\def\E{{\mathcal E}}

\def\V{{\mathcal V}}

\def\Y{{\mathcal Y}}
\def\X{{\mathcal X}}

\def\x{\times}                   
\def\cong{\simeq}

\def\+{\oplus}                   
\def\*{\otimes}                  

\def\Spec{\operatorname{Spec}}
\def\TT{\operatorname{T}}

\def\Aut{\operatorname{Aut}}
\def\Ker{\operatorname{Ker}}
\def\Id{\operatorname{id}}

\def\Hom{\operatorname{Hom}}

\def\Def{\operatorname{Def}}

\begin{document}

\title{A note on  deformations of regular embeddings} 

\author{C. Ciliberto}
\address{Ciro Ciliberto, Dipartimento di Matematica, Universit\`a di Roma Tor Vergata, Via della Ricerca Scientifica, 00173 Roma, Italy}
\email{cilibert@mat.uniroma2.it}

\author{F. Flamini}
\address{Flaminio Flamini, Dipartimento di Matematica, Universit\`a di Roma Tor Vergata, Via della Ricerca Scientifica, 00173 Roma, Italy}
\email{flamini@mat.uniroma2.it}

\author{C. Galati}
\address{Concettina Galati, Dipartimento di Matematica, Universit\`a della Calabria, via P. Bucci, cubo 31B, 87036 Arcavacata di Rende (CS), Italy}
\email{galati@mat.unical.it}

\author{A. L. Knutsen}
\address{Andreas Leopold Knutsen, Department of Mathematics, University of Bergen, Postboks 7800,
5020 Bergen, Norway}
\email{andreas.knutsen@math.uib.no}

\begin{abstract}  
In this paper we give a description of the first order deformation space
of a regular embedding  $X\hookrightarrow Y$ of reduced algebraic schemes. We  compare our result with results of Ran (in particular  \cite[Prop. 1.3]{ran}).
\end{abstract}

\maketitle

\begin{flushright}
\emph{This paper is dedicated to Philippe Ellia\\ on the occasion of his sixtieth birthday.}
\end{flushright}

\section*{Introduction}
The deformation theory of morphisms $f: X\to Y$ between schemes over an algebraically closed field $\bf k$ is an important subject, with a lot of applications, and there is a vast literature concerning it (e.g., see References). In general, there is a cotangent complex in a derived category of right bounded complexes giving rise to $\bf k$-vector spaces $T^i_{f}$, with  $T^1_{f}$ the tangent space to the deformation functor of $f$ and $T^2_{f}$ the corresponding obstruction space. When $Y$ is smooth and $f$ is a regular embedding, the description of $T^1_{f}$ and $T^2_{f}$ can be found in \cite [\S 3]{ser}. The aim of this note is to extend this description to the more general  case in which $Y$ is singular, but reduced. 
The main result  is  Proposition \ref {main-result}. The techniques we use are standard in deformation theory.  

This problem has been studied also by Z. Ran in \cite {ran},  where the author makes a proposal  for a \emph{classifying space} (Ran's terminology) for first order deformations of any morphism $f: X\to Y$. In \S \ref {ssec:ran} we compare Ran's result with ours (in the regular embedding case) and we observe that Ran's space maps to $T^1_{f}$, but, in general, the map is not an isomorphism. 
 
The reason we got involved in this topic, has been our work on the universal Severi variety of nodal curves on the moduli spaces of polarised $K3$ surfaces and on the related moduli map, cf. \cite{cfgk1,cfgk2}. In \S \ref {sec:sev} we make some considerations on this subject and explain how our results can be used to attack the study of moduli problems for Severi varieties by degeneration arguments (as we did in  \cite {cfgk2}). We think that these ideas can be usefully applied to investigate still unexplored areas like moduli problems for Severi varieties on  surfaces other than $K3$s, e.g., Enriques surfaces. 


\subsection*{Acknowledgements} 
We express deep gratitude to E. Sernesi for useful discussions on the subject of this paper. We also benefited from conversations with M. Lehn and M. Kemeny.  
Finally we would like to thank the referee for his detailed report. The first three authors  have been supported by the GNSAGA of Indam and by the PRIN project ``Geometry of projective varieties'', funded by the Italian MIUR. 

\subsection*{Terminology }
For  deformation theory we refer to  \cite{ser, hart2,GLS}.  We will
use the same notation and terminology as in \cite{ser}. In particular, 
 a closed embedding $\nu:X\hookrightarrow Y$ of algebraic schemes 
will mean a closed immersion as in  \cite[p.85]{hart}. For a closed embedding
$\nu:X\hookrightarrow Y$  we will set $\Omega^1_Y|_X:=\nu^*(\Omega^1_Y)$.
 

\section{First order deformations of closed regular embeddings}

\subsection {Preliminaries} Let $X$ and $Y$ be reduced (noetherian and separated) algebraic schemes over an algebraically closed field $\bf k$ and let 
$\nu:X\hookrightarrow Y$ be a \emph{regular closed embedding} of codimension $r$, i.e., $X$, as a subscheme of $Y$, is locally defined 
by a regular sequence  (cf. \cite[App. D]{ser}). We want to study 
first order deformations of $\nu$ without assuming, as in \cite[\S 3.4.4]{ser}, that $Y$ is smooth.

A first order deformation $\tilde\nu:\mathcal X\to\mathcal Y$ of $\nu$ is a cartesian diagram  
\begin{eqnarray*}\label{first-order-def}
\xymatrix{    X\,\, \ar@{^{(}->}[r]^i\ar@{^{(}->}[d]^\nu &  \mathcal X\ar@{^{(}->}[d]^{\tilde\nu}\\
 Y\ar[d] \,\,\ar@{^{(}->}[r]^j &  \mathcal Y\ar[d]\\
 \Spec(\bf k) \ar@{^{(}->}[r]^s&   \Spec(\bf k[\epsilon]) }
\end{eqnarray*}
where $\mathcal X$ and $\mathcal Y$ are flat over $\Spec(\bf k[\epsilon])$ (cf. \cite[Def. 3.4.1]{ser}). In particular $\mathcal X$ and $\mathcal Y$ are infinitesimal deformations of $X$ and $Y$ respectively. 
The notions of isomorphic, trivial and locally trivial deformations of $\mathcal X$, $\mathcal Y$ and $\nu$ are given as usual.

\begin{remark}\begin{inparaenum}
\item  In the diagram above $\nu$ is a closed embedding (cf. \cite[Note 3, p. 185]{ser}).\\
\item 
If $\tilde\nu:\mathcal X\hookrightarrow\mathcal Y$ is a first order deformation of $\nu$ and $\Y$ is  trivial, then $\tilde\nu$ is a regular embedding.
This follows by flatness using \cite[Cor. A.11]{ser}.\\
\item If $Y$ is a Cohen-Macaulay scheme and $\tilde\nu:\mathcal X\hookrightarrow\mathcal Y$ is a first order deformation of $\nu$, then $\tilde\nu$ is again a regular embedding. Indeed, since $Y$ is Cohen-Macaulay, then $\Y$ is Cohen-Macaulay (cf. \cite[Thm. 24.5]{ma} or \cite[\S 10.129]{stacks}) and the result follows by \cite[Thm. 9.2]{hart2}. 
\end{inparaenum}
\end{remark}

From now on we will denote by $\rm{Def}_{\nu}$ and $\rm{Def}_{Y}$ 
(resp. $\rm{Def'}_{\nu}$ and $\rm{Def'}_{Y}$) the functors of
deformations (resp. first order locally trivial deformations) of $\nu$ and $Y$ respectively, 
whose properties are described in the aforementioned  references, especially in \cite{ser}.
For the functors $\rm{Def}_{\nu}$ and $\rm{Def}_{Y}$  there is a cotangent complex in a derived category of right bounded complexes giving rise to the $\bf k$-vector spaces $T^i_{\nu}$ and   
$T^i_{Y}$. Then  $T^1_{\nu}$ and   
$T^1_{Y}$ are the tangent spaces to $\rm{Def}_{\nu}$ and $\rm{Def}_{Y}$ and $T^2_{\nu}$ and   
$T^2_{Y}$ the corresponding obstruction spaces (cf. \cite[Thm. C.5.1, Cor. C.5.2]{GLS} and related references as \cite{lis, flen, flen1,Ill, Ill1, buc}).  The aim of this section is to describe the vector space $T^1_{\nu}\cong \rm{Def}_{\nu}({\bf k}[\epsilon])$.
 
Let $\mathcal I$ be a coherent locally free sheaf on $Y$. We will use the  standard identifications of vector spaces  
\begin{equation}\label{standard-identifications1}
\Ex_{{\kk}}(Y,\mathcal I)\simeq\Ext^1_{\O_Y}(\Omega^1_{Y},\mathcal I),
\end{equation}
where  $\Ex_{{\kk}}(Y,\mathcal I):=\Ex(Y/\Spec({\kk}),\mathcal I)$ is the space of (infinitesimal) extensions of $Y$ by $\mathcal I$,
and
\begin{equation}\label{standard-identifications}
\rm{Def}_{Y}({\kk}[\epsilon])\simeq\Ex_{{\kk}}(Y,\O_{Y})\simeq\Ext^1_{\O_Y}(\Omega^1_{Y},\O_{Y})\,\,\quad \mbox{and}\,\,\quad 
\rm{Def'}_{Y}({\kk}[\epsilon])\simeq H^1(Y,\Theta_{Y}),
\end{equation}
where $\Theta_Y={\mathcal Hom}_{\O_Y}(\Omega^1_Y,\O_Y)$, while $H^0(Y,\Theta_{Y})$
is the space of infinitesimal automorphisms of $Y$ (cf. \cite[Thms. 1.1.10 and 2.4.1 and Prop. 2.6.2]{ser}).  

\begin{remark}\label{conormal}
The deformation theory of a regular closed embedding is easier than the one of an arbitrary closed embedding because of the properties of 
its conormal sequence. Indeed, if $\nu:X\hookrightarrow Y$ is a regular closed  embedding  of codimension $r$ and $I$ is the ideal sheaf of $X\subseteq Y$, 
then $I/I^2$ is a  locally free sheaf on $X$ of rank $r$. This implies that the conormal sequence 
\begin{equation}\label{eq:conorm}
\xymatrix{ 0 \ar[r] & I/I^2  \ar[r]  & \Omega{{_Y^1}|}_{X} \ar[r]^{\beta} & \Omega_X^1 \ar[r] & 0}
\end{equation}
of $\nu$ is exact on the left.
\end{remark}

\subsection {The description of the first order deformation space} 

 Recall the natural morphisms 
\[\xymatrix{\mu:   \Ext^1_{\O_Y}(\Omega^1_{Y},\O_{Y}) \ar[r] & \Ext^1_{\O_X}(\Omega^1_{Y}|_{X},\O_{X})} \; \; \mbox{and} \; \; \xymatrix{\lambda:\Ext^1_{\O_X}(\Omega^1_{X},\O_{X}) \ar[r] & \Ext^1_{\O_X}(\Omega^1_{Y}|_{X},\O_{X}),}\]
where:\\
\begin{inparaenum}
\item  [$\bullet$] $\mu$  restricts an extension 
$0   \to\O_{Y} \to \E \to \Omega^1_{Y} \to 0$ to $X$;\\
\item  [$\bullet$] $\lambda$ sends an extension 
$      0   \to   \O_{X} \to  \E \to      \Omega^1_{X} \to 0$
to the fiber product $\E \x_{\Omega^1_{X}} \Omega^1_Y|{_X}$ via $\beta$ in  \eqref{eq:conorm}. 
\end{inparaenum}

\begin{proposition}\label{main-result}
Let $\nu:X\hookrightarrow Y$ be a regular closed embedding of reduced algebraic schemes. Then
 the first order deformation space of $\nu$ is isomorphic to the fiber product
\begin{equation}\label{H}
\xymatrix{   \rm{Def}_{\nu}({\bf k}[\epsilon])\simeq\Ext^1_{\O_X}(\Omega^1_{X},\O_{X})\times_{\Ext^1_{\O_X}(\Omega^1_{Y}|_{X}, \O_{X})}
\Ext^1_{\O_Y}(\Omega^1_{Y},\O_{Y})\ar[r]^{\hspace{3,5cm} p_Y}\ar[d]_{p_X}& \Ext^1_{\O_Y}(\Omega^1_{Y},\O_{Y})\ar[d]^{\mu}\\
\Ext^1_{\O_X}(\Omega^1_{X},\O_{X})\ar[r]^\lambda &\Ext^1_{\O_X}(\Omega^1_{Y}|_{X},\O_{X}).}
\end{equation}
\end{proposition}

\begin{proof}
Let $s\in \Ext^1_{\O_X}(\Omega^1_{X},\O_{X})$ and $t\in\Ext^1_{\O_Y}(\Omega^1_{Y},\O_{Y})$ correspond to first order deformations
$\mathcal X\to \Spec({\bf k}[\epsilon])$ and $\mathcal Y\to \Spec({\bf k}[\epsilon])$ of $X$ and $Y$ respectively, with closed embeddings
 $i:X\hookrightarrow\X$ and $j:Y\hookrightarrow\Y$, cf. \eqref{standard-identifications}. With this interpretation, we have
\[ 
s:=\Big(0   \to\O_{Y} \to\Omega^1_{\Y}|_{Y} \to \Omega^1_{Y} \to 0\Big)\quad  \stackrel{\mu}\longmapsto  \quad\mu(s):=\Big(0   \to \O_{X} \to\Omega^1_{\Y}|_{X} \to \Omega^1_{Y}|_{X} \to 0\Big)\]
and
\[ t:=\Big(0   \to   \O_{X} \to    \Omega^1_{\X}|_{X} \to    \Omega^1_{X}\to 0\Big) \quad  \stackrel {\lambda}\longmapsto\quad  \lambda(t):= \Big( 0 \to \O_X \to  \Omega^1_{\X}|_X \x_{\Omega^1_X} \Omega^1_{Y}|_X \to \Omega^1_Y|_{X}\Big),\]
so that we have a commutative diagram 
\begin{equation}\label{lambda(epsilon)}
\xymatrix{    &     &     0\ar[d]  &    0\ar[d]  & \\
  & 0 \ar[r]  \ar[d]   & \O_{X} \ar@{=}[r] \ar[d]   & \O_{X} \ar[r] \ar[d]  & 0\\
  0   \ar[r]   & I/I^2 \ar[r] \ar@{=}[d]    & \Omega^1_{Y}|_{X}\times_{\Omega^1_{X}}\Omega^1_{\X}|_{X}  \ar[r]^{\hspace{0.6cm}L} \ar[d]^\alpha  
  & \Omega^1_{\X}|_{X} \ar[r] \ar[d]^M & 0\\
   0   \ar[r]   &I/I^2 \ar[r]  \ar[d]   & \Omega^1_{Y}|_{X} \ar[r]^\beta \ar[d]   & \Omega^1_{X} \ar[r] \ar[d] & 0\\ 
   &   0  &     0  &    0  &}
\end{equation}
where $I$ is the ideal sheaf of $X$ in $Y$, as in Remark \ref{conormal}.
\color{black}

We also  observe that \color{black} $j\nu:X\hookrightarrow Y\hookrightarrow\Y$
is a regular closed embedding and $\Y$ is reduced.   The conormal exact sequences (see Remark \ref{conormal}) of $\nu$, $j$ and $j\nu $ 
fit in the  commutative diagram
\begin{equation}\label{conormal-nu-i}
\xymatrix{    &     &     0\ar[d]  &    0\ar[d]  & \\
  0   \ar[r]   & \O_X  \ar@{=}[d]  \ar[r]   & {\mathcal I}/{\mathcal I}^2 \ar[r] \ar[d]   &  I/I^2\ar[r] \ar[d]  & 0\\
  0   \ar[r]   & \O_X \ar[r]     & \Omega^1_{\mathcal Y}|_{X} \ar[r] \ar[d]^{\gamma }   & \Omega^1_{Y}|_{X} \ar[r] \ar[d] & 0\\
      &   & \Omega^1_{X}\ar@{=}[r] \ar[d]   & \Omega^1_{X}  \ar[d] & \\ 
   &     &     0  &    0  &}
\end{equation}
where $\mathcal I$ is the ideal sheaf of $X$ in $\Y$.

 To prove the proposition, we want to prove that $\lambda (s)=\mu(t)$ if and only if there exists a closed embedding
\begin{equation}\label{inclusion}
\xymatrix{  \mathcal X\,\ar@{^{(}->}[r]^{\tilde\nu}\ar[dr]& \mathcal Y\ar[d]\\
& \Spec(\bf k[\epsilon])}
\end{equation}
restricting to $\nu:X\to Y$ over $\Spec(\bf k)$, i.e. such that $\tilde\nu i=j\nu$.

 First  assume that $\lambda (s)=\mu(t)$, i.e.,  there exists  a commutative diagram 
\begin{equation}\label{lambda=mu}
\xymatrix{    0   \ar[r]   & \O_X \ar[r] \ar@{=}[d]    & \Omega^1_{\mathcal Y}|_{X} \ar[r] \ar[d]^\simeq  & \Omega^1_{Y}|_{X} \ar[r] \ar@{=}[d] & 0\\
 0 \ar[r]   & \O_X  \ar[r]   &\Omega^1_{Y}|_{X}\times_{\Omega^1_{X}}\Omega^1_{\X}|_{X}\ar[r]  & \Omega^1_{Y}|_{X} \ar[r]  &  0.}
\end{equation}
By  \eqref{lambda=mu}, the map $\beta\alpha=ML$  in \eqref{lambda(epsilon)}
can be identified with  the map $\gamma$ in \eqref{conormal-nu-i},  whose kernel is 
 $${\mathcal I}/{\mathcal I}^2\simeq \ker(\beta\alpha)=\ker(ML)\simeq I/I^2\oplus\O_{X}.$$ 
In particular the conormal sequence of $X$ in $\Y$, which is the central vertical sequence in \eqref{conormal-nu-i}, 
 fits in the  commutative diagram 
\begin{equation}\label{lambda(epsilon)1}
\xymatrix{    & 0\ar[d]     &     0\ar[d]  &    0\ar[d]  & \\
  0   \ar[r]   & I/I^2\ar[r]  \ar@{=}[d]   & I/I^2\oplus\O_{X} \ar[r] \ar[d]   & \O_{X} \ar[r] \ar[d]  & 0\\
  0   \ar[r]   & I/I^2 \ar[r] \ar[d]    & \Omega^1_{\mathcal Y}|_{X} \ar[r]^L \ar[d]^{\beta\alpha=\gamma}   & \Omega^1_{\X}|_{X} \ar[r] \ar[d]^M & 0\\
      &0 \ar[r]   & \Omega^1_{X}\ar@{=}[r] \ar[d]   & \Omega^1_{X} \ar[r] \ar[d] & 0\\ 
   &     &     0  &    0  &}
\end{equation}
Since  $I/I^2\oplus  \O_{X}$ is locally free,  by  \eqref {lambda(epsilon)1}  we obtain the new diagram
\begin{equation}\label{lambda(epsilon)2}
\xymatrix{    & 0\ar[d]     &     0\ar[d]  &    0\ar[d]  & \\
  0   \ar[r]   & I/I^2\ar[r]  \ar@{=}[d]   & I/I^2\oplus\O_{X} \ar[r] \ar[d]   & \O_{X} \ar[r] \ar[d]  & 0\\
  0   \ar[r]   & I/I^2 \ar[r] \ar[d]    & \O_{\mathcal Y} \ar[r] \ar[d]   & \O_{\X} \ar[r] \ar[d] & 0\\
      &0 \ar[r]   & \O_{X}\ar@{=}[r] \ar[d]   & \O_{X} \ar[r] \ar[d] & 0\\ 
   &     &     0  &    0  &}
\end{equation}
where the  map $ \O_{\Y}\to  \O_{\X}$ (which provides the desired embedding $\X\subseteq \Y$) is induced by   $L$ in \eqref {lambda(epsilon)1}, and by the isomorphisms (see  \cite[Proof of Thm. 1.1.10]{ser})
\[
 \O_{\Y}\simeq\Omega^1_{\mathcal Y}|_{X}\times_{\Omega^1_{X}}\O_X\quad \text{and} \quad \O_{\X}\simeq\Omega^1_{\mathcal X}|_{X}\times_{\Omega^1_{X}}\O_X
\]
where the fiber  products are between the derivation $d:\O_X\to\Omega^1_X$ and the conormal maps.

Conversely, assume that there is a closed embedding  \eqref{inclusion} such that $\tilde\nu i=j\nu$. Then one obtains a diagram like \eqref{lambda(epsilon)},
with $\Omega^1_{Y}|_{X}\times_{\Omega^1_{X}}\Omega^1_{\X}|_{X}$  replaced by $\Omega^1_{\Y}|_X$. Using the universal property of the 
fiber product, one deduces an isomorphism of extensions as in \eqref{lambda=mu}, ending the proof of the proposition. 
\end{proof}

\begin{remark}\label{lambda-in-general} Suitable versions of the maps $\lambda$ and $\mu$ can be defined even if the embedding $\nu$ is not regular. However our proof of Proposition \ref{main-result} does not extend, as it is, to this more general case, because the kernel of the conormal sequence of $X$ in $Y$ is no longer locally free.
\end{remark}

\subsection {Comments} 

To  better understand the maps  $\lambda$ and $\mu$, and the related geometry, observe that they fit in the following diagram with exact rows and columns:
\begin{equation} \label{lambda-mu}
\end{equation}
\vspace{-1cm}
\[
\tiny\begin{array}{ccccccc}
 & 0& & 0 & &\\
 &\downarrow & &\downarrow & &\\
 &\Hom_{\O_Y}(\Omega^1_{Y},I) &= &\Hom_{\O_Y}(\Omega^1_{Y},I) & &\\
 & \downarrow & &\downarrow & &\\
 &H^0(\Theta_Y) &= & H^0(\Theta_Y)  & &\\
 &\downarrow & &\downarrow & &\\
 &\Hom_{\O_X}(\Omega^1_{Y}|_X,\O_{X}) & = &\Hom_{\O_X}(\Omega^1_{Y}|_X,\O_{X}) & &\\
 &\downarrow & &\downarrow & &\\
 &\Ext^1_{\O_Y}(\Omega^1_{Y},I) & =&\Ext^1_{\O_Y}(\Omega^1_{Y},I) & &\\
 &\downarrow & &\downarrow & &\\
0   \to   H^0(\Theta_{X})\to \Hom_{\O_X}(\Omega^1_{Y}|_X,\O_{X}) \to  
  H^0(N_{X/Y}) \longrightarrow &\rm{Def}_{\nu}({\bf k}[\epsilon]) &\stackrel{p_Y}{\longrightarrow} &\Ext^1_{\O_Y}(\Omega^1_{Y},\O_Y) &\longrightarrow & H^1(N_{X/Y})\\
\hspace{0,5cm}\parallel\hspace{2cm} \parallel\hspace{2cm} \parallel  & \hspace{0.38cm} \downarrow^{p_X} & &\hspace{0.17cm}\downarrow^\mu & & \parallel \\
0   \to   H^0(\Theta_{X})\to \Hom_{\O_X}(\Omega^1_{Y}|_X,\O_{X}) \to  
  H^0(N_{X/Y}) {\longrightarrow} &\Ext^1_{\O_X}(\Omega^1_{X},\O_{X})& \stackrel{\lambda}{\longrightarrow} &  \Ext^1_{\O_X}(\Omega^1_{Y}|_X,\O_{X})& \longrightarrow & H^1(N_{X/Y})\\
   &\downarrow & &\downarrow & &\downarrow \\
 &\Ext^2_{\O_Y}(\Omega^1_{Y},I)&=&\Ext^2_{\O_Y}(\Omega^1_{Y},I)& & \Ext^2_{\O_X}(\Omega^1_{X},\O_{X}) \\
    &\downarrow & &\downarrow & &\downarrow \\
 &\Ext^2_{\O_Y}(\Omega^1_{Y},\O_Y) &=&\Ext^2_{\O_Y}(\Omega^1_{Y},\O_Y) & & \Ext^2_{\O_X}(\Omega^1_{Y}|_X,\O_{X})\\
 &\downarrow&&\downarrow&  &\downarrow  \\
 &\Ext^2_{\O_Y}(\Omega^1_{Y},\O_X)&=&\Ext^2_{\O_Y}(\Omega^1_{Y},\O_X)& & \vdots \\
&\vdots && \vdots &\\
 \end{array}
\]
\normalsize
where the lower  row arises from the  conormal sequence of $\nu:X\hookrightarrow Y$ and the second column  from  $0\to I\to\O_Y\to\O_X\to 0.$ We denoted by $N_{X/Y}={\mathcal Hom}_{\O_X}(I/I^2,\O_X)$  the normal sheaf of $X$ in $Y$,  we used the standard isomorphisms $H^i(X,N_{X/Y})\simeq \Ext_{\O_X}^i(I/I^2,\O_X)$ (using that $I/I^2$ is locally free) and one checks that $\Ext^i_{\O_X}(\Omega^1_{Y}|_X,\O_{X})\simeq \Ext^i_{\O_Y}(\Omega^1_{Y},\O_{X})$,
for $i=0,1$. The space $H^0(X,N_{X/Y})$ consists of first order
deformations of $X$ as a subscheme of (the fixed scheme) $Y$, while $H^1(X,N_{X/Y})$ is the corresponding obstruction space (because $X$ is regularly embedded in $Y$), and the map $H^0(N_{X/Y}) \rightarrow \rm{Def}_{\nu}({\bf k}[\epsilon])$  is the obvious 
morphism. 

The diagram \eqref {lambda-mu}  and the cotangent braid \cite[p.~446]{GLS} suggest that $\Hom_{\O_X}(\Omega^1_{Y}|_X,\O_{X})$ should be the first order deformation space of $\nu$ preserving $X$ and $Y$ (cf. \cite[\S 3.4.1]{ser}), while $\Ext^1_{\O_X}(\Omega^1_{Y}|_X,\O_{X})$ should be the corresponding obstruction space. This is the case if $Y$ is smooth (see \cite [Prop. 3.4.2]{ser}). The first fact follows from the following more general result.

\begin{proposition}\label{fixed-domain-target1}
Let $f:X\to Y$ be a morphism of reduced algebraic schemes and let ${\rm Def}_{X/f /Y}$ be the deformation functor of $f$ preserving $X$ and $Y$. Then \begin{equation}\label{fixed-domain-target-def}
{\rm Def}_{X/f /Y}(\kk(\epsilon))\simeq\Hom_{\O_X}(f^*\Omega^1_{Y},\O_{X}).
\end{equation}
If  $\Hom_{\O_X}(f^*\Omega^1_{Y},\O_{X})$ has finite dimension (in particular if $X$ is projective), then   ${\rm Def}_{X/\nu /Y}$ is pro-representable.
\end{proposition}

\begin{proof}
Let $j:\Gamma\hookrightarrow X\times Y$ be the embedding of the graph of $f$ in $X\times Y$ and $q:X\times Y\to Y$ and 
$p:X\times Y\to X$ the natural projections. Then, by arguing  as in  step (i) of the proof of \cite[Prop. 3.4.2]{ser}, 
we find  a natural isomorphism of functors between  ${\rm Def}_{X/f /Y}$ and the local Hilbert functor $H^{X\times Y}_\Gamma.$
In particular, by \cite[Prop. 3.2.1]{ser}, we have  
$${\rm Def}_{X/\nu /Y}(\kk(\epsilon))\simeq H^{X\times Y}_\Gamma(\kk(\epsilon))\simeq H^0(\Gamma,N_{\Gamma |X\times Y})\simeq\Hom_{\O_{\Gamma}}(\I /{\I}^2,\O_{\Gamma}),$$
where $\I$ is the ideal sheaf of $\Gamma$ in $X\times Y$.
Now observe that $\Gamma$ and $X$ are isomorphic via $pj$. Hence $\Gamma$ is reduced and the conormal  sequence 
of $\Gamma$ in $X\times Y$ can be written as 
$$
\xymatrix{0 \ar[r] & \A \ar[r] &  \I /{\I}^2 \ar[r] & \Omega^1_{X\times Y}|_{\Gamma} \ar[r] & \Omega^1_{\Gamma} \ar[r] &  0,}
$$
where $\A$ is a torsion sheaf on $\Gamma$. Moreover we have  the (split) exact sequence
$$
\xymatrix{0 \ar[r] &  q^*\Omega^1_{Y} \ar[r] & \Omega^1_{X\times Y} \ar[r] &  p^*\Omega^1_{X} \ar[r] &  0.}
$$
 By restricting this   to $\Gamma$ one finds the  exact sequence
$$
\xymatrix{0 \ar[r] &  j^*q^*\Omega^1_{Y} \ar[r] &  j^*\Omega^1_{X\times Y}=\Omega^1_{X\times Y}|_{\Gamma} \ar[r] &  j^*p^*\Omega^1_{X} \ar[r] &  0.}
$$
Since $j^*p^*\Omega^1_{X}$ and $\Omega^1_\Gamma$ are isomorphic, one obtains an isomorphism   $j^*q^*\Omega^1_{Y}\cong (\I /{\I}^2)/ \A$. Moreover, $j^*q^*\Omega^1_{Y}=(pj)^*f^*\Omega^1_{Y}$.
It follows that
$$
{\rm Def}_{X/\nu /Y}(\kk(\epsilon))\simeq  \Hom_{\O_{\Gamma}}(\I /{\I}^2,\O_{\Gamma})\simeq\Hom_{\O_{\Gamma}}((\I /{\I}^2)/\A,\O_{\Gamma})
\simeq   \Hom_{\O_X}(f^*\Omega^1_{Y},\O_{X}).
$$
\normalsize
The last statement of the proposition follows by the isomorphism of functors $H^{X\times Y}_\Gamma\simeq {\rm Def}_{X/f/Y}$ and the fact that,
under the hypothesis,
$H^{X\times Y}_\Gamma$ is pro-representable (cf. \cite[Cor. 3.2.2]{ser}).
\end{proof}

\begin{remark}
The proof of Proposition \ref{fixed-domain-target1} adapts the proof of \cite[Prop. 3.4.2]{ser} to the singular case.
Note that, if $Y$ is singular, the graph of $f:X\to Y$ is in general  no longer a regular embedding, even if $f$ is a regular embedding. Therefore
$$
H^1(\Gamma, N_{\Gamma|X\times Y})
\simeq H^1({\mathcal Hom}_{\O_X}(f^ *\Omega^1_{Y},\O_{X}))\subseteq \Ext^1_{\O_X}(f^ *\Omega^1_{Y},\O_{X}) 
$$
is not necessarily the obstruction space for $H^{X\times Y}_\Gamma\simeq {\rm Def}_{X/f /Y}$. 
\end{remark}

\subsection{Comparison with results by Z. Ran}\label {ssec:ran}

In \cite {ran} there is a proposal for a \emph{classifying space} for first order deformations of a morphism $f: X\to Y$, with $X,Y$ any pair of schemes. Given $f: X\to Y$ there are two obvious maps
\[
\xymatrix{ \delta_0: f^ *\O_Y \ar[r] &  \O_X} \quad \text{and}\quad \xymatrix{ \delta_1: f^ *\Omega^ 1_Y \ar[r] & \Omega^ 1_X.}
\]
In \cite {ran} one constructs vector spaces $\Ext^ i(\delta_1,\delta_0)$, $i\geqslant 0$, fitting in the long exact sequence
$$
\xymatrix{ 0 \ar[r] &  {\rm Hom} (\delta_1,\delta_0) \ar[r] &
{\rm Hom}_{\O_{X}}(\Omega^1_{X},\O_{X})\oplus{\rm Hom}_{\O_{Y}}(\Omega^1_{Y},\O_{Y})
\ar[r]^{\hspace{1.3cm}\varphi_0} & {\rm Hom}_{\O_{X}}(f^ *\Omega^1_{Y},\O_{X}) \ar[r]^{\hspace{1.3cm}\partial} &  \\  
 \ar[r]^{\hspace{-0.7cm}\partial} & \Ext^1(\delta_1,\delta_0) \ar[r] &
\Ext^1_{\O_X}(\Omega^1_{X},\O_{X})\oplus\Ext^1_{\O_Y}(\Omega^1_{Y},\O_{Y})
\ar[r]^{\hspace{1.3cm}\varphi_1} & 
\Ext^1_{\O_X}(f^ *\Omega^1_{Y},\O_{X}) \ar[r] & \dots}
$$
 where
\[
\varphi_0(\alpha,\beta)=\alpha \delta_1-\delta_0f^ *\beta
\]
and the map $\varphi_1$ is defined accordingly. The result in \cite {ran} is that $\Ext^1(\delta_1,\delta_0)$ is the classifying space in question (see \cite [Prop. 3.1]{ran}). 

We remark that $\Ext^1(\delta_1,\delta_0)$ does not, in general,  coincide with ${\rm Def}_{f}({\bf k}[\epsilon])$. Indeed, let us consider the  case in which $f:X\to Y$ is a regular embedding. Then, by \eqref {lambda-mu}, one has $\varphi_1=\lambda-\mu$.  Therefore 
\[
{\rm Def}_{f}({\bf k}[\epsilon])\cong {\rm Ker}(\varphi_1).
\]
We now provide an example where the map $\partial$ is non--zero,  which shows that  $\Ext^1(\delta_0,\delta_1)$ surjects onto ${\rm Def}_{f}({\bf k}[\epsilon])$ but is not isomorphic to it.

\begin{example}\label{ex:ran} {\rm Let $\pi: Y\to Z$ be a smooth surjective  morphism with $Z$ smooth of positive dimension, $Y$ irreducible such that $h^ 0(Y,\Theta_Y)=0$ and irreducible fibres all isomorphic to a fixed $X$ such that $h^ 0(X,\Theta_X)=0$. Then $N_{X/Y}\cong \O_X^ {\dim(Z)}$ and the coboundary map $H^ 0(X,N_{X/Y})\to H^ 1(X,\Theta_X)$ is zero, hence $h^ 0(X,{\Theta_Y}|_X)=\dim(Z)>0$ and $\partial$ is non--zero.

This situation is easy to cook up: it suffices to take $Y=X\times Z$ and $h^ 0(X,\Theta_X)=h^ 0(Z,\Theta_Z)=0$. 
 }
\end{example}

\section{Remarks on deformations of nodal curves on normal crossing surfaces}\label{sec:sev} Let $S$ be a connected surface with (at most) normal crossing singularities and  $i:C\hookrightarrow S$ be the regular embedding of a (reduced) nodal curve. Let $N\subset S$ be the  length-$\delta$  scheme of nodes of $C$ lying on the smooth locus of $S$ and  $\pi:Y\to S$ be the blowing-up at $N$. Denote by $X$ the proper transform of $C$ in $Y$. Then $\phi=\pi|_{X}:X\to S$ is the partial normalization of $C=\phi(X)$ at the nodes on the smooth locus of $S$. We will  assume that $X$ is connected. We will set $g:=p_a(X)$ (since $X$ is connected, then $g\geqslant 0$). 
 Denote by $\nu:X\hookrightarrow Y$ the embedding of $X$ in $Y$ and by 
$\rm{Def}_{\phi}(\bf k[\epsilon])$ the first order deformation space of $\phi$. 
\begin{lemma}\label{kodaira}
There exists a natural isomorphism 
$$\rm{Def}_{\nu}(\bf k[\epsilon])\simeq\rm{Def}_{\phi}(\bf k[\epsilon]).$$  
\end{lemma}

\begin{proof}
One has an exact sequence 
\begin{equation}\label{C1}
\xymatrix{
0 \ar[r] & \bf k^{2\delta} \ar[r] & \Ext^1_{\O_{Y}}( \Omega^ 1_{Y},\O_{Y})\ar[r] & 
\Ext^1_{\O_S}(\Omega^ 1_S,\O_S)\ar[r] & 0}
\end{equation}
inducing a morphism $\rm{Def}_{\nu}(\bf k[\epsilon])\rightarrow\rm{Def}_{\phi}(\bf k[\epsilon])$,
which is an isomorphism.
\end{proof}

Now assume $H$ is a polarization on $S$ and assume that there is an irreducible component $\mathcal B$ 
of a moduli scheme
parametrizing isomorphism classes of polarized surfaces with no worse singularities than normal crossings, such that $(S,H)\in\mathcal B$ and the general point
of $\mathcal B$ corresponds to a pair $(S',H')$, with $S'$ smooth and irreducible.

One sees that there exists (at least locally) a scheme $\V_{m,\delta}$, called
the $(m,\delta)$--\emph{universal Severi variety}, endowed with a morphism 
\begin{equation*}\label{eq:phiVK2}
\xymatrix{ \phi_{m,\delta}: \V_{m,\delta} \ar[r] & \mathcal B.}
\end{equation*} 
The points in $\V_{m,\delta}$ are pairs $(S',C')$ with $(S',H')\in \mathcal B$, and  $C'\in |mH'|$ nodal, with exactly $\delta$ nodes on the smooth locus of $S'$, and connected  normalization at these $\delta$ nodes (in \cite[\S~2]{cfgk1} we treated the special case of $K3$ surfaces). 

Let $\V$ be an irreducible component of $\V_{m,\delta}$ and assume  that for $(S',C')\in \V$ general, the  normalization of $C'$ at the $\delta$ nodes is stable. Then one has the obvious {\em moduli map} 
\begin{equation*}\label{eq:modmap2}
\xymatrix{ \psi_{m,\delta} :\V \ar@{-->}[r] & {\overline {\mathcal M}}_{g},} 
\end{equation*}
where  ${\overline {\mathcal M}}_g$ is the Deligne--Mumford compactification of the moduli space of smooth, genus $g$ curves. Given $(S,C)\in \mathcal V$, if the normalization $X$ of $C$ at the $\delta$ nodes on the smooth locus of $S$ is stable, then 
$\psi_{m,\delta}$ is defined at $(S,C)$.  In particular, from Lemma \ref{kodaira} and diagram \eqref{lambda-mu}, we obtain:

\begin{corollary}\label{lem:tg}
There are  natural identifications 
$$
{\rm T}_{(S, C)}\V_{m,\delta}\simeq \rm{Def}_{\phi}({\bf k[\epsilon]})\simeq \rm{Def}_{\nu}({\bf k[\epsilon]}).
$$
Moreover, if $X$ is stable, then the map  
$$p_X:\rm{Def}_{\nu}({\bf k[\epsilon]})\simeq {\rm T}_{(S, C)}\V_{m,\delta}\longrightarrow \rm{Ext}^1_{\O_X}(\Omega_{X},\O_{X})\simeq T_{[X]}{\overline {\mathcal M}}_g$$
in \eqref{H} (and \eqref{lambda-mu}) is the differential $d_{(S,C)}\psi_{m,\delta}$  of $\psi_{m,\delta}$ at $(S,C)$. In particular, if $\rm{Ext}^1_{\O_Y}(\Omega^ 1_{Y}(X),\O_{Y})=0$ (resp. $\rm{Ext}^2_{\O_Y}(\Omega^ 1_{Y}(X),\O_{Y})=0$), then $d_{(S,C)}\psi_{m,\delta}$ is injective (resp. surjective).
\end{corollary}

\begin{remark}
When $S$ is a smooth surface,  the tangent space to $\V_{m,\delta}$ at $(S, C)$
coincides with the space of first order locally trivial deformations of $i:C\hookrightarrow S.$
In particular,  by \cite[Prop. 3.4.17]{ser} we know that  ${\rm T}_{(S, C)}\V_{m,\delta}\simeq H^1(S,T_S\langle C\rangle)$, where $T_S\langle C\rangle$ is the sheaf defined in \cite[(3.56)]{ser}, and  ${\rm{Ext}}^i_{\O_S}(\Omega_{S}(C),\O_{S})\simeq H^{i}(S,T_S(-C))$, for $i=0,1,2,$
as  observed in \cite{fkps} in the case $S$ is a  $K3$ surface. 
\end{remark}

In general the computation of the cohomology groups appearing in diagram \eqref{lambda-mu}, hence in the stament of Lemma \ref {lem:tg}, is difficult.  A possible approach to this problem is by degeneration, as we explain now.

 Since the surface $Y$ has normal crossing, there is on $Y$ a locally free sheaf $\Lambda_{Y}^1$ (cf. \cite[Thm. (3.2)]{fri}) encoding important information about deformations of $Y$. If $f:\mathcal Y\to\Delta$ is any semi-stable deformation
of $Y=f^{-1}(0)$ \cite[(1.12)]{fri} with general fibre $Y_t$, there is a locally free sheaf 
$\Omega_{{\mathcal Y}/\Delta}({\rm{log}}\,Y)$ on
$\mathcal Y$ such that  $\Omega_{{\mathcal Y}/\Delta}({\rm{log}}\,Y)|_{Y_t}\simeq
\Omega^1_{Y_t}$ and
\begin{equation}\label{lambda-uno}
\Lambda_{Y}^1\simeq
\Omega_{{\mathcal Y}/\Delta}({\rm{log}}\,Y)|_{Y}
\end{equation}
(see \cite[\S 3]{fri}, \cite[\S 2]{cfgk2} and related references). 

Let  $\mathcal X\subset\mathcal Y\to\Delta$ be a flat family of curves with fibres $X_t$, for $t\in \Delta$ and 
$X_0=X$ as above. Then, by flatness and semicontinuity, one has 
\begin{equation*}
\dim ({\rm{Ext}}^i_{\O_{Y}}(\Lambda^1_{Y}\otimes \O_Y(X),\O_{Y}))\geqslant \dim({\rm{Ext}}^i_{\O_{Y_t}}(\Omega^1_{Y_t}\otimes \O_{Y_t}(X_t),\O_{Y_t})), \quad \text {for $t\in \Delta$ general}.
\end{equation*}
Hence, in order to prove that the moduli map is generically of maximal rank, it suffices to 
prove vanishing theorems for  ${\rm{Ext}}^i_{\O_{Y}}(\Lambda^1_{Y}(X),\O_{Y})$, for $i=1$ or $i=2$, on $Y$, which in certain cases may be easier to obtain than the vanishings of ${\rm{Ext}}^i_{\O_{Y_t}}(\Omega^1_{Y_t}\otimes \O_{Y_t}(X_t),\O_{Y_t})$
on the general $Y_t$. 

This approach has proved to be useful in the case of $K3$ surfaces (see \cite {cfgk2}). 

%
%

\newpage
%
%


\section*{Errata corrige to \\ 
``A note on  deformations of regular embeddings''}

The main result of the paper \cite{cfgk3} (Proposition 1.3) is wrongly stated.
Nevertheless the proof of Proposition 1.3 and Proposition 1.5 provide
a complete description of ${{\Def}}_{\nu}({\kk}[\epsilon])$ and the paper needs only the corrections below.  
 
\section*{Corrections}\label{sect: Corrections}
 
$\bullet$ The statement of Proposition 1.3 has to be replaced by the following, which is exactly what is proved.
 
 \begin{main-proposition}\label{old-main-result}
Let $\nu:X\hookrightarrow Y$ be a regular closed embedding of reduced algebraic schemes and let $ {\Def}_{X/\nu/Y}$ be the deformation functor of $\nu$ preserving $X$ and $Y$  (cf. \cite[\S 3.4.1]{ser}). Then
there exists a surjective morphism $\Phi$  from $ {\Def}_{\nu}({\kk}[\epsilon])$ to the fiber product 

\begin{equation}\renewcommand{\theequation}{4}\label{H}
\xymatrix{ 
{{\Def}}_{\nu}({\bf k}[\epsilon])  \ar[d]_{\Phi} &\\
   \Ext^1_{\O_X}(\Omega^1_{X},\O_{X})\times_{\Ext^1_{\O_X}(\Omega^1_{Y}|_{X}, \O_{X})}
\Ext^1_{\O_Y}(\Omega^1_{Y},\O_{Y})\ar[r]^{\hspace{2,9cm} p_Y}\ar[d]_{p_X}& \Ext^1_{\O_Y}(\Omega^1_{Y},\O_{Y})\ar[d]^{\mu}\\
 \Ext^1_{\O_X}(\Omega^1_{X},\O_{X})\ar[r]^\lambda &\Ext^1_{\O_X}(\Omega^1_{Y}|_{X},\O_{X})}
\end{equation}
whose kernel is the image of the natural map $\xymatrix{\Delta: {\Def}_{X/\nu /Y}(\kk[\epsilon]) \ar[r] & 
{\Def}_{\nu}({\bf k}[\epsilon])}$.
\end{main-proposition}

Recalling that 
\begin{equation}
\label{eq:prop1.5}
{\Def}_{X/\nu/Y}(\kk[\epsilon])\simeq \Hom_{\O_X}(\nu^*\Omega^1_{Y},\O_{X})=\Hom_{\O_X}(\Omega^1_{Y}|_X,\O_{X}), \tag{$\dagger$}
\end{equation} 
by Proposition 1.5, we obtain the following result describing $ {{\Def}}_{\nu}({\bf k}[\epsilon])$, which is now to be considered the main result of the paper. (In the statement, the map $\xymatrix{\beta: \Omega^1_Y|_X 
\ar[r] & \Omega^1_X}$ is the one in the conormal sequence.)
 
 \begin{main-result}
 Let $\nu:X\hookrightarrow Y$ be a regular closed embedding of reduced algebraic schemes.  Then there exists a long exact sequence

\begin{footnotesize} 
\[\xymatrix{ 
0 \ar[r] &  \Hom_{\O_X}(\Omega^1_{X},\O_{X})\times_{\Hom_{\O_X}(\Omega^1_{Y}|_{X}, \O_{X})}\Hom_{\O_Y}(\Omega^1_{Y},\O_{Y})\ar[r] & \Hom_{\O_X}(\Omega^1_{X},\O_{X}) \x \Hom_{\O_Y}(\Omega^1_{Y},\O_{Y}) \ar@{-}[r]^{\hspace{2.5cm}\Theta} &} \]
\[
\xymatrix{
\ar[r]^{\hspace{-1.4cm}\Theta} & \Hom_{\O_X}(\Omega^1_{Y}|_X,\O_{X}) \ar[r]^{\hspace{0.7cm}\Delta} &  {\Def}_{\nu}({\bf k}[\epsilon]) \ar[r]^{\hspace{-3cm}\Phi} &  \Ext^1_{\O_X}(\Omega^1_{X},\O_{X})\times_{\Ext^1_{\O_X}(\Omega^1_{Y}|_{X}, \O_{X})}\Ext^1_{\O_Y}(\Omega^1_{Y},\O_{Y})\ar[r] &  0}
\]
\end{footnotesize}

where the map $\Theta$ is given by $\Theta(\xi,\eta)=\xi \circ \beta -\eta|_X$.
\end{main-result}
 
\begin{proof}
  The second row of the above exact sequence follows from (the above version of) Proposition 1.3 and \eqref{eq:prop1.5}. 

By the definition of $\Hom_{\O_X}(\nu^*\Omega^1_{Y},\O_{X})$ and $ {{\Def}}_{\nu}({\bf k}[\epsilon])$ (cf. \cite[p. 158 and p. 177]{ser}), an element mapped to zero by $\Delta$ corresponds to a first order deformation $$\tilde\nu:X\times {\Spec}({\bf k}[\epsilon])\to Y\times {\Spec}({\bf k}[\epsilon])$$ of $\nu$  that is trivializable. More precisely, denoting by $H_X\subset {\Aut} ( X \x \Spec ( {\bf k}[\epsilon]))$ the space of automorphisms restricting to the identity on the closed fibre and similarly for  $H_Y\subset {\Aut} (Y \x \Spec ( {\bf k}[\epsilon]))$, there exist $\alpha\in H_X$ and $\beta\in H_Y$, such that $$\alpha\circ (\nu\times {\Id}_{\Spec ({\bf k}[\epsilon])})\circ \beta=\tilde\nu.$$
Then one obtains a natural map $H_X\times H_Y\to {\Def}_{X/\nu/Y}(\kk[\epsilon])$ whose image is the kernel of $\Delta$. By \eqref{eq:prop1.5} and the well-known isomorphisms $H_X\simeq \Hom_{\O_X}(\Omega^1_{X},\O_{X})$ and $H_Y\simeq \Hom_{\O_Y}(\Omega^1_{Y},\O_{Y})$ (cf. \cite[Lemma 1.2.6]{ser}), this map may be identified with $\Theta$. The kernel of $\Theta$ is by definition as in the statement.
\end{proof}

$\bullet$ In the first column of diagram (11) the vector space $ {\Def}_{\nu}({\bf k}[\epsilon])$ must be replaced by the quotient  $ {\Def}_{\nu}({\bf k}[\epsilon])/{\Im}( \Delta)$. 

\vspace{0.3cm}

$\bullet$ The paragraph ``We remark that $\Ext^1(\delta_1,\delta_0)\dots$ not isomorphic to it.'' in \S 1.4 has to be replaced by the following:

``We remark that $\Ext^1(\delta_1,\delta_0)$ coincides with ${{\Def}}_{\nu}({\bf k}[\epsilon])$ in the case when $f: X \to Y$ is a regular embedding. By  (11), with $ {\Def}_{\nu}({\bf k}[\epsilon])$ replaced by $ {\Def}_{\nu}({\bf k}[\epsilon])/{\Im} (\Delta)$, one has $\varphi_1=\lambda-\mu$. Therefore,
  \[
\Ext^1_{\O_X}(\Omega^1_{X},\O_{X})\times_{\Ext^1_{\O_X}(\Omega^1_{Y}|_{X}, \O_{X})}\Ext^1_{\O_Y}(\Omega^1_{Y},\O_{Y})\cong {\Ker}(\varphi_1),
\]
$\Delta$ coincides with $\partial$ and $\Theta$ with $\varphi_0$. Example 1.7 below gives an instance where $\partial=\Delta$ is nonzero.'' 

 \vspace{0.3cm}

$\bullet$ In the proof of Lemma 2.1, the exact sequence (13) is not exact on the left, but this does not affect the proof.

 \vspace{0.3cm}

$\bullet$ Replace the statement of Corollary 2.2 by the following: 

\begin{main-corollary}\label{lem:tg} There is a natural surjective map
$$
\tau: {\TT}_{(S, C)}\V_{m,\delta}\longrightarrow {\Def}_{\phi}({\bf k[\epsilon]})\simeq {\Def}_{\nu}({\bf k[\epsilon]}).
$$

Moreover, if $X$ is stable, then the differential of the moduli map  of $\psi_{m,\delta}$ at $(S,C)$ factors as

$$d_{(S,C)}\psi_{m,\delta}: {\TT}_{(S, C)}\V_{m,\delta}\stackrel{\tau}\longrightarrow{\Def}_{\nu}({\bf k[\epsilon]}) \longrightarrow  {\Def}_{\nu}({\bf k[\epsilon]})/ {\Im}(\Delta) \stackrel{\mathit p_X}\longrightarrow {\Ext}^1_{\O_X}(\Omega_{X},\O_{X})\simeq T_{[X]}{\overline {\mathcal M}}_g,$$
where $p_X$ is the map appearing in the correct version of (11).

In particular,  if ${\Ext}^2_{\O_Y}(\Omega^ 1_{Y}(X),\O_{Y})=0$, then $d_{(S,C)}\psi_{m,\delta}$ is surjective;
if 
$${\Ext}^1_{\O_Y}(\Omega^ 1_{Y}(X),\O_{Y})={\Hom}_{\O_X}(\Omega^ 1_{Y}|_X,\O_{X})= {\Hom}_{\O_Y}(\Omega^ 1_{Y},\O_{Y})=0$$ then $d_{(S,C)}\psi_{m,\delta}$ is injective.
\end{main-corollary}

\vspace{0.3cm}

$\bullet$ At the end of  Remark 2.3, add  "In this case, using the above notation, one has $ {\Hom}_{\O_Y}(\Omega^ 1_{Y},\O_{Y})=H^ 0(Y,T_Y)=0$ and moreover,  by \cite[(4) in proof of Prop. 1.2]{ck}, ${\Hom}_{\O_X}(\Omega^ 1_{Y}|_X,\O_{X})=
H^0(X, {T_Y}_{|_X})=0$."

\subsection*{Acknowledgements} 
We wish to thank Marco Manetti for having kindly pointed out to us that the statement of Proposition 1.3 in \cite{cfgk3}
was wrong and provided precious informations on related topics.


\end{document}